\documentclass[review,3p]{elsarticle}
\usepackage{lineno, hyperref}
\usepackage{amsfonts} 
\usepackage{graphicx} 
\usepackage{amsmath} 
\usepackage{color}
\usepackage{ntheorem}
\newtheorem{thm11}{Theorem}[section]
\newtheorem{theorem}[thm11]{Theorem}
\newtheorem{lemma}[thm11]{Lemma}
\newtheorem{proposition}[thm11]{Proposition}

\newtheorem{definition}[thm11]{Definition}
\newtheorem{example}[thm11]{Example}
\newenvironment{proof}{\emph{Proof. }}{\hspace*{\fill}$\Box$\par\vskip2ex}
\numberwithin{equation}{section}
\usepackage{graphicx}
\graphicspath{{Figures/}}

\journal{}
\makeatletter
\def\ps@pprintTitle{%
	\let\@oddhead\@empty
	\let\@evenhead\@empty
	\def\@oddfoot
	{\hbox to \textwidth%
		{\ifnopreprintline\relax\else
			\@myfooterfont%
			\ifx\@elsarticlemyfooteralign\@elsarticlemyfooteraligncenter%
			\hfil\@elsarticlemyfooter\hfil%
			\else%
			\ifx\@elsarticlemyfooteralign\@elsarticlemyfooteralignleft%
			\@elsarticlemyfooter\hfill{}%
			\else%
			\ifx\@elsarticlemyfooteralign\@elsarticlemyfooteralignright%
			{}\hfill\@elsarticlemyfooter%
			\else%
			\hfill\@date
			\fi%
			\fi%
			\fi%
			\fi%
		}
	}%
	\let\@evenfoot\@oddfoot
}
\makeatother

\usepackage{tikz}
\usetikzlibrary{decorations.markings}
\tikzstyle{vertex}=[circle, draw, inner sep=0pt, minimum size=6pt]
\usepackage{caption}
\usepackage{subcaption}

\newcommand{\RS}{\operatorname{RS}}
\newcommand{\RSV}{\operatorname{RSV}}

\begin{document}

\begin{frontmatter}

\title{A method for constructing graphs with the same resistance spectrum}


\author{Si-Ao Xu}

\author{Huan Zhou}

\fntext[myFunding]{Funding:Xiang-Feng Pan was supported by Natural Science Foundation of Anhui Province under Grant No. 2108085MA02 and University Natural Science Research Project of Anhui Province under Grant No. KJ2020A0001.
}

\author{Xiang-Feng Pan \fnref{myFunding} \corref{mycorrespondingauthor}}
\cortext[mycorrespondingauthor]{Corresponding author}
\ead{xfpan@ahu.edu.cn}

\address{School of Mathematical Sciences, Anhui University, Hefei, Anhui, 230601, P.R. China}

\begin{abstract}
Let $G=(V(G),E(G))$ be a graph with vertex set $V(G)$ and edge set $E(G)$.
The resistance distance $R_G(x,y)$ between two vertices $x,y$ of $G$ is defined to be the effective resistance between the two vertices in the corresponding electrical network in which each edge of $G$ is replaced by a unit resistor. The resistance spectrum $\RS (G)$ of a graph $G$ is the multiset of the resistance distances of all pairs of vertices in the graph. This paper presents a method for constructing graphs with the same resistance spectrum. It is obtained that for any positive integer $k$, there exist at least $2^k$ graphs with the same resistance spectrum. Furthermore, it is shown that for $n \geq 10$, there are at least $2(n-9) p(n-9)$ pairs of graphs of order $n$ with the same resistance spectrum, where $p(n-9)$ is the number of partitions of the integer $n-9$.
\end{abstract}


\end{frontmatter}

\section{Introduction}\label{sec:introduction}
In 1993, Klein and Randi\'{c} \cite{KR1993} introduced the concept of resistance distance based on the theory of electrical networks. The \emph{resistance distance} $R_{G} (x, y)$ between two vertices $x$ and $y$ of a graph $G$ is defined as the effective resistance of the two points in the corresponding electrical network, which the electrical network is attained from $G$ by replacing each edge of the graph with a unit resistor.

The \emph{resistance spectrum} $\RS (G)$ of a graph $G$ is defined as the multiset of the resistance distances of all pairs of vertices in the graph. The resistance spectrum of a graph had been initially used to solve the graph isomorphism problem by Baxter \cite{Baxter1999RS22} who conjectured that two graphs are isomorphic if and only if their resistance spectra are identical. However, this conjecture had been quickly disproved after some counterexamples \cite{Baxter1999Counterexample26, Rickard1999Counterexample23} were found.

All nonisomorphic simple graphs with no more than $8$ vertices are determined by their resistance spectra.
However, there are exactly $11$ and $49$ pairs of nonisomorphic graphs, each pair of which shares the same resistance spectrum, among all simple graphs with $9$ and $10$ vertices, respectively \cite{WeissteinRSEquivalent}. In addition, a number of other pairs with the same resistance spectrum but different structures have been discovered.
Figure \ref{fig:Counterexamples13} illustrates $13$ such pairs of nonisomorphic graphs, whose resistance spectra are given in Table \ref{table:Resistance spectra of the graphs}. The first $11$ pairs of graphs here are the ones with 9 vertices. The 2nd, 3rd, and 12th, 13th pairs were discovered by Baxter \cite{Baxter1999Counterexample26} and Rickard \cite{Rickard1999Counterexample23}, respectively.
A graph $G$ is said to be $DRS$ if it is determined by the resistance spectrum, that is, there is no nonisomorphic graph with the same resistance spectrum as $G$; conversely, if there is a nonisomorphic graph such that it has the same resistance spectrum as $G$, then $G$ is said to be $non$-$DRS$.

\begin{figure} [htp]
	\centering
	\includegraphics[width=0.7\textwidth]{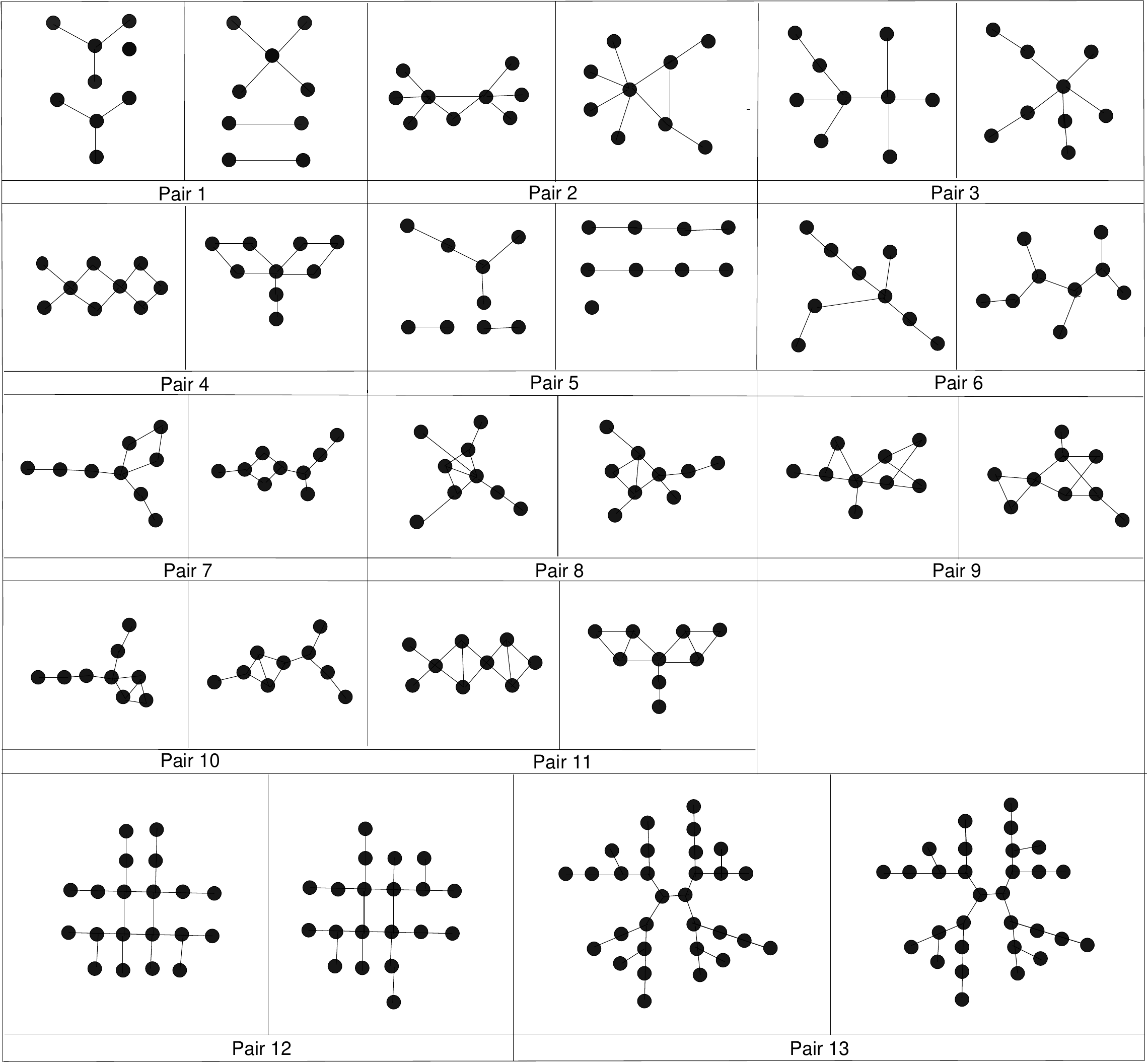}
	\caption{$13$ pairs of nonisomorphic graphs with the same resistance spectrum \cite{WeissteinRSEquivalent}.} \label{fig:Counterexamples13}
\end{figure}

\begin{table} [htp]\scriptsize
	\caption{Resistance spectra of the graphs \cite{WeissteinRSEquivalent}.} \label{table:Resistance spectra of the graphs}
	\centering
	\begin{tabular} {|c|l|}
		\hline Graph pair & Resistance spectrum\\
		\hline 1 	& $[1]^{6}$ 	 $[2]^{6}$ 	 $[+\infty]^{24}$\\
		\hline 2	& $[2/3]^{3}$ 	 $[1]^{6}$ 	 $[5/3]^{12}$ $[2]^{6}$ $[8/3]^{9}$\\
		\hline 3 	& $[1]^{8}$ 	 $[2]^{13}$ 	 $[3]^{12}$ 	 $[4]^{3}$\\
		\hline 4 	& $[3/4]^{8}$ 	 $[1]^{6}$ 	 $[3/2]^{4}$ $[7/4]^{8}$ $[2]^{4}$ $[11/4]^{4}$ $[3]^{2}$\\
		\hline 5 	& $[1]^{6}$ 	 $[2]^{4}$ 	 $[3]^{2}$ 	 $[+\infty]^{24}$\\
		\hline 6 	& $[1]^{8}$ 	 $[2]^{10}$ 	 $[3]^{10}$ 	 $[4]^{6}$ $[5]^{2}$\\
		\hline 7 	& $[1/2]^{1}$ 	 $[5/8]^{4}$ $[1]^{6}$ 	 $[3/2]^{2}$ $[13/8]^{8}$ $[2]^{4}$ $[5/2]^{1}$ $[21/8]^{6}$ $[3]^{2}$ $[29/8]^{2}$\\
		\hline 8 	& $[3/4]^{4}$ 	 $[1]^{7}$ 	 $[7/4]^{4}$ $[2]^{6}$ $[11/4]^{4}$ $[3]^{5}$ $[15/4]^{2}$ $[4]^{3}$ $[5]^{1}$\\
		\hline 9 	& $[2/3]^{10}$ $[1]^{5}$ 	 $[4/3]^{4}$ $[5/3]^{10}$ $[2]^{2}$ $[7/3]^{2}$ $[8/3]^{3}$\\
		\hline 10 	& $[1/2]^{1}$ 	 $[5/8]^{4}$ $[1]^{6}$ 	 $[13/8]^{4}$ $[2]^{6}$ $[21/8]^{4}$ $[3]^{5}$ $[29/8]^{2}$ $[4]^{3}$ $[5]^{1}$\\
		\hline 11 	& $[1/2]^{2}$ 	 $[5/8]^{8}$ $[1]^{4}$ 	 $[5/4]^{4}$ $[13/8]^{8}$ $[2]^{4}$ $[21/8]^{4}$ $[3]^{2}$\\
		\hline 12 	& $[3/4]^{4}$ 	 $[1]^{18}$ 	 $[7/4]^{16}$ 	 $[2]^{22}$ $[11/4]^{32}$ $[3]^{24}$ $[15/4]^{32}$ $[4]^{18}$ $[19/4]^{16}$ $[5]^{8}$\\
		\hline 13 	& $[1]^{29}$ 	 $[2]^{38}$ 	 $[3]^{50}$ 	 $[4]^{64}$ $[5]^{78}$ $[6]^{82}$ $[7]^{64}$ $[8]^{26}$ $[9]^{4}$\\
		\hline
	\end{tabular}
	\\\noindent \footnotesize The resistance distances are shown in ascending order. $[a/b]^{n}$ denotes $n$ occurrences of the fraction $a/b$.
\end{table}

\section{Preliminary knowledge}\label{sec:Preliminary}
In this paper, we only consider simple undirected graphs. For undefined notations and terminologies, see the book by Bondy and Murty \cite{BM2008graph}.

A partition of a positive integer $t$ is a multiset of positive integers that sum to $t$. We denote the number of partitions of $t$ by $p (t)$. For example, since $\{5\}, \{4, 1\}, \{3, 2\}, \{3, 1, 1\}, \{2, 2, 1\}, \{2, 1, 1, 1\}$, $\{1, 1, 1, 1, 1\}$ are all the partitions of $5$, $p(5)=7$.

\begin{definition}
	Let $A=\{a_1, a_2, \ldots, a_p\}$ and $B=\{b_1, b_2, \ldots, b_q\}$ be two different partitions of a positive integer $n$. We say $A$ and $B$ are \emph{of equal sums of squares} if 
	$$\sum_{i=1}^{p} a_i^2 = \sum_{i=1}^{q} b_i^2.$$
\end{definition}

\begin{example}
	$A:=\{3, 3\}, B:=\{4, 1, 1\}$ are of equal sums of squares.
\end{example}

\begin{proposition}
	Let $A$ and $B$ be two different partitions of a positive integer $n$, where $A=\{a_1, a_2, \ldots, a_p\}$ and $B=\{b_1, b_2, \ldots, b_q\}$. If $A$ and $B$ are of equal sums of squares, then
	$$\sum_{1 \leq i < j \leq p} a_i a_j = \sum_{1 \leq i < j \leq q} b_i b_j.$$
	
\end{proposition}

\begin{definition}
	Let $G$ be a graph with $V(G) = \{g_1, g_2, \ldots, g_n\}$.
	Let $S$ be a subset of $V(G)$, where $S=\{g_{k_1}, g_{k_2}, \ldots, g_{k_s}\}$, $1 \leq s \leq n$ and $1 \leq k_1 < \cdots < k_s \leq n$.
	Let $A =\{a_1, a_2, \ldots, a_p\}$ be a partition of a positive integer $t$, where $p \leq
	s$.
	Let $H_1, H_2, \ldots, H_t$ be $t$ graphs, where $V(H_i) =\{h_{i, 1}, h_{i, 2}, \ldots, h_{i, n_i} \}$ for $i = 1, 2, \ldots, t$.
	Let $\mathcal{H} = (H_1, H_2, \ldots, H_t)$ and $T=\{h_{1,t_1},\ldots,h_{t,t_t}\}$.
	The graph $G (S, A, \mathcal{H},T)$ is constructed from $G$ and $\mathcal{H}$ by identifying $g_{k_1}, h_{1, t_1}, h_{2, t_2}, \ldots, h_{a_1, t_{a_1}}$, identifying $g_{k_i}, h_{\sum_{j=1}^{i-1} a_j+1, t_{\sum_{j=1}^{i-1} a_j+1}}, h_{\sum_{j=1}^{i-1} a_j+2, t_{\sum_{j=1}^{i-1} a_j+2}}, \ldots$, $h_{\sum_{j=1}^{i} a_j, t_{\sum_{j=1}^{i} a_j}}$, where $i=2, 3, \ldots, p$.
	
\end{definition}

\begin{example}
	Let $G$ be a cycle of length 3 with vertex set $\{g_1,g_2,g_3\}$, $S=\{g_2, g_3\}$ and $A=\{3, 3\}$. Let $\mathcal{H} = (H_1, H_2, \ldots, H_6)$ and $T=\{h_{1,1},h_{2,1},\ldots,h_{t,1}\}$, where $H_i$ is a path of length $1$ with $V(H_i)=\{h_{i, 1}, h_{i, 2}\}$ for $i=1, 2, \ldots, 6$.
	The graph $G (S, A, \mathcal{H},T)$ is depicted in Figure \ref{fig:C3-3-3}.
\end{example}

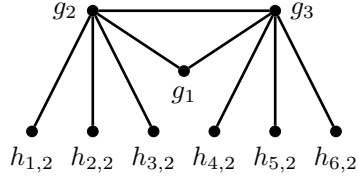
\begin{figure}[htp]
	\centering
	\tikzstyle{vertex}=[circle, draw, fill=black, inner sep=0, minimum size=4pt]
	\tikzstyle{line}=[line width=1pt,color=black]
	\begin{tikzpicture}[scale = 0.8]
	\node[vertex, label=below:$g_1$] (1) at (3.,2.5) {};
	\node[vertex, label=left:$g_2$] (2) at (1.5,3.5) {};
	\node[vertex, label=right:$g_3$] (3) at (4.5,3.5) {};
	\node[vertex, label=below:$h_{1, 2}$] (4) at (0.5,1.5) {};
	\node[vertex, label=below:$h_{2, 2}$] (5) at (1.5,1.5) {};
	\node[vertex, label=below:$h_{3, 2}$] (6) at (2.5,1.5) {};
	\node[vertex, label=below:$h_{4, 2}$] (7) at (3.5,1.5) {};
	\node[vertex, label=below:$h_{5, 2}$] (8) at (4.5,1.5) {};
	\node[vertex, label=below:$h_{6, 2}$] (9) at (5.5,1.5) {};
	
	\draw [line] (1) -- (2);
	\draw [line] (1) -- (3);
	\draw [line] (2) -- (3);
	\draw [line] (2) -- (4);
	\draw [line] (2) -- (5);
	\draw [line] (2) -- (6);
	\draw [line] (3) -- (7);
	\draw [line] (3) -- (8);
	\draw [line] (3) -- (9);
	\end{tikzpicture}
	\caption{$G (S, A, \mathcal{H},T)$.}\label{fig:C3-3-3}
\end{figure}

We denote by $\RSV (G, v)$ the multiset of the resistance distances between $v$ and other vertices of $G$,
that is, $\RSV (G, v)=\{R_G(v, u) \mid u \in V(G)\setminus \{v\}\}$.

For two nonempty multisets $A$ and $B$, both consisting of real numbers, we define the sum of $A$ and $B$ as the following multiset:
\begin{equation*}
A+B = \{a+b \mid a \in A, b \in B\}.
\end{equation*}

\begin{lemma}\cite{KR1993}\label{lemma:RDcut}
	Let $x$ be a cut vertex of a connected graph $G$. Let $u$ and $v$ be two vertices belonging to different components after $x$ is deleted from $G$. Then $R_G(u, v)=R_G(u, x)+R_G(x, v)$.
\end{lemma}

\begin{proposition}\label{proposition:1}
	Let $G_1, G_2, H_1$ and $H_2$ be four graphs. Let $V(G_i)=\{g_{i, 1}, \ldots, g_{i, n}\}$, $V(H_i)=\{h_{i, 1}, \ldots$, $h_{i, m}\}$, $S_i=\{g_{i, 1}\}$, $T_i=\{h_{i, 1}\}$, $\mathcal{H}_i=(H_i)$, where $i=1, 2$, and $A=\{1\}$.
	Let $G=G_1(S_1, A, \mathcal{H}_1, T_1)$ and $H=G_2(S_2, A, \mathcal{H}_2, T_2)$.
	If $\RS (G_1)=\RS(G_2)$, $\RSV (G_1, g_{1, 1})=\RSV (G_2, g_{2, 1})$, $\RS (H_1)=\RS (H_2)$ and $\RSV (H_1, h_{1, 1})=\RSV (H_2, h_{2, 1})$,
	then $\RSV (G, g_{1, 1})=\RSV (H, g_{2, 1})$ and $\RS (G)=\RS (H)$.
\end{proposition}
\begin{proof}
	
	By Lemma \ref{lemma:RDcut}, we have
	\begin{gather*}
	\RS (G) =\RS (G_1) \cup \RS (H_1) \cup (\RSV (G_1, g_{1, 1})+ \RSV (H_1, h_{1, 1})),\\
	\RS (H) =\RS (G_2) \cup \RS (H_2) \cup (\RSV (G_2, g_{2, 1})+ \RSV (H_2, h_{2, 1})),\\
	\RSV (G, g_{1, 1}) = \RSV (G_1, g_{1, 1}) \cup \RSV (H_1, h_{1, 1})
	\end{gather*}
	and
	$$\RSV (H, g_{2, 1}) = \RSV (G_2, g_{2, 1}) \cup \RSV (H_2, h_{2, 1}).$$
	The results can be reached by a simple examination.
\end{proof}

For two graphs $G$ and $H$, if $\RS (G)=\RS (H)$ and there is a vertex $g$ of $G$ and a vertex $h$ of $H$, such that $\RSV (G, g)=\RSV (H, h)$, then we say that $G$ and $H$ hold relation $\mathcal{U}$ with respect to vertices $g$ and $h$, denoted by $G$-$g$-$\mathcal{U}$-$h$-$H$. Sometimes, we say that $G$ holds relation $\mathcal{U}$ with $H$ instead for simplicity while $G$-$g$-$\mathcal{U}$-$h$-$H$ is abbreviated as $G\mathcal{U}H$.


\begin{example} \label{exam:T1T2}
	If $T_1$ and $T_2$ are two graphs of $9$ verteices as shown in Figure \ref{fig:T1T2}, then $T_1$ holds relation $\mathcal{U}$ with $T_2$.
	
	By a simple calculation,
	$\RS (T_1)=\RS (T_2) =\{ [5]^1, [4]^{3}, [\frac{15}{4}]^{2}, [3]^5, [\frac{11}{4}]^4,[2]^6,[\frac{7}{4}]^4,[1]^7,[\frac{3}{4}]^4 \}$ and\\
	$\RSV (T_1, t_{1, 3})=\RSV (T_2, t_{2, 3}) =\{\frac{15}{4}, \frac{11}{4}, \frac{11}{4}, \frac{7}{4},\frac{7}{4} , 1, \frac{3}{4}, \frac{3}{4}\}$.
\end{example}
\begin{figure}[htpb!]
	\centering
	\tikzstyle{vertex2}=[circle, draw, fill=black, inner sep=0, minimum size=4pt]
	\tikzstyle{line}=[line width=1pt,color=black]
	\begin{subfigure}[b]{0.49\textwidth}  %
		\centering
		\begin{tikzpicture}
		\node[vertex2, label=below:$t_{1, 2}$] (1) at (0, 0) {};
		\node[vertex2, label=below:$t_{1, 3}$] (2) at (1,1) {};
		\node[vertex2, label=below:$t_{1, 4}$] (3) at (1,-1) {};
		\node[vertex2, label=below:$t_{1, 5}$] (4) at (2, 0) {};
		\node[vertex2, label=below:$t_{1, 6}$] (5) at (3, 1) {};
		\node[vertex2, label=below:$t_{1, 7}$] (6) at (3, -1) {};
		\node[vertex2, label=below:$t_{1, 8}$] (7) at (4, 1) {};
		\node[vertex2, label=below:$t_{1, 9}$] (8) at (4, -1) {};
		\node[vertex2, label=below:$t_{1, 1}$] (9) at (5, -1) {};
		
		\draw [line] (1) -- (2);
		\draw [line] (1) -- (3);
		\draw [line] (2) -- (4);
		\draw [line] (3) -- (4);
		\draw [line] (4) -- (5);
		\draw [line] (4) -- (6);
		\draw [line] (5) -- (7);
		\draw [line] (6) -- (8);
		\draw [line] (8) -- (9);
		\end{tikzpicture}
		\caption{$T_1$}
	\end{subfigure}
	\begin{subfigure}[b]{0.49\textwidth}
		\centering
		\begin{tikzpicture}
		\node[vertex2, label=below:$t_{2, 2}$] (1) at (0, 1) {};
		\node[vertex2, label=below:$t_{2, 1}$] (2) at (1, 1) {};
		\node[vertex2, label=right:$t_{2, 3}$] (3) at (2, 2) {};
		\node[vertex2, label=below:$t_{2, 4}$] (4) at (2, 0) {};
		\node[vertex2, label=below:$t_{2, 5}$] (5) at (3, 1) {};
		\node[vertex2, label=right:$t_{2, 6}$] (6) at (4, 1) {};
		\node[vertex2, label=right:$t_{2, 7}$] (7) at (5, 2) {};
		\node[vertex2, label=below:$t_{2, 8}$] (8) at (5, 0) {};
		\node[vertex2, label=below:$t_{2, 9}$] (9) at (6, 0) {};
		
		\draw [line] (1) -- (2);
		\draw [line] (2) -- (3);
		\draw [line] (2) -- (4);
		\draw [line] (3) -- (5);
		\draw [line] (4) -- (5);
		\draw [line] (5) -- (6);
		\draw [line] (6) -- (7);
		\draw [line] (6) -- (8);
		\draw [line] (8) -- (9);
		\end{tikzpicture}
		\caption{$T_2$}
	\end{subfigure}
	\caption{Two graphs $T_1$ and $T_2$ holding the relation $\mathcal{U}$.}
	\label{fig:T1T2}
\end{figure}
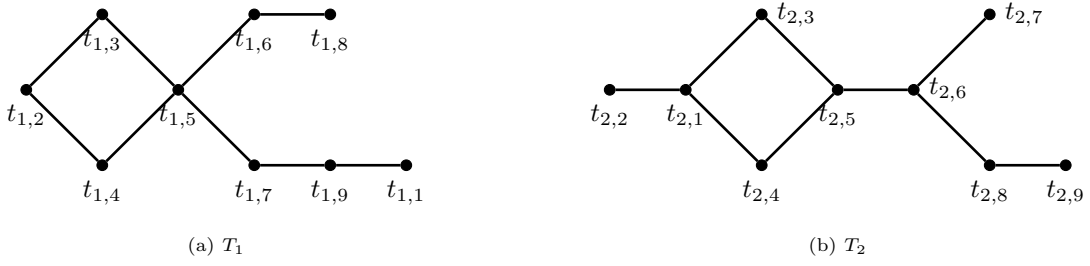

\begin{proposition} \label{prop:GSAFGSAH}
	Let $G$ be a graph with vertex set $\{g_1, g_2, \ldots, g_n\}$, and $S=\{g_{k_1}, g_{k_2}, \ldots, g_{k_s}\}$, where $1 \leq s \leq n$ and $1 \leq k_1 < \cdots < k_s \leq n$.
	Let $F_1, F_2, \ldots, F_t,
	H_1, H_2, \ldots, H_t$ be graphs with $F_i$-$f_{i,1}$-$ \mathcal{U}$-$h_{i,1}$-$H_i$,  $f_{i,1}\in V(F_i) $ and $h_{i,1}\in V(H_i)$ for $i=1, 2, \ldots, t$. Let $\mathcal{F}=(F_1, F_2, \ldots, F_t)$, $\mathcal{H}=(H_1, H_2, \ldots, H_t)$, $T_1=\{f_{1,1},\ldots,f_{t,1}\}$ and $T_2=\{h_{1,1},\ldots,h_{t,1}\}$. Let $A$ be a partition of positive integer $t$,
	where $A=\{a_1, a_2, \ldots, a_p\}$, $1 \leq p \leq s \leq n$. Then $G (S, A, \mathcal{F}, T_1)$ and $G (S, A, \mathcal{H},T_2)$ have the same resistance spectrum.
\end{proposition}
\begin{proof}
	Since $F_1, F_2, \ldots, F_t, H_1, H_2, \ldots, H_t$ be graphs with $F_i$-$f_{i,1}$-$\mathcal{U}$-$h_{i,1}$-$H_i$ for $i=1, 2, \ldots, t$, we have $\RS (H_i)=\RS (F_i), i=1, 2, \ldots, t$, and $\RSV(H_i, h_{i, 1})=\RSV(F_i, f_{i, 1})$.
	Then by Proposition \ref{proposition:1}, repeatedly, we have\\
	$$\RS (G (S, A, \mathcal{F}, T_1))=\RS (G (S, A, \mathcal{H}, T_2)).$$
\end{proof}

\begin{theorem}
	For any positive integer $k$, there exist at least $2^k$ graphs with the same resistance spectrum.
\end{theorem}
\begin{proof}
	Let $L_k$ denote the set of all $k$-dimensional row vectors consisting of elements $1$ or $2$. Clearly, $|L_k| = 2^k$.
	Let $T_1$ and $T_2$ be defined as in Example \ref{exam:T1T2}.
	Let $G_{k,p,q} = u_1 u_2 \cdots u_p v_1 v_2 \cdots v_k w_1 w_2 \cdots w_q $ be a path of length $k+p+q-1$, where $q \geq p+1 \geq 8$. Let $S = \{v_1, v_2, \ldots, v_k\}$ and $A=\{[1]^k\}$.
	For any element $I$ in $L_k$, let $I = (I_1, I_2, \ldots, I_k)$, $\mathcal{H_I} = (H_1, H_2, \ldots, H_k)$, where $H_i$ is isomorphic to $T_{I_i}$ and the vertex $h_{i,1}$ is identical to $T_{I_i,3}$ under some isomorphism $\theta$ between $H_i$ and $T_{I_i}$. 
Let $T=\{h_{1,1}, h_{2,1}, \ldots, h_{t,1}\}$ and $G_{k,p,q}^I = G_{k,p,q}(S,A,\mathcal{H_I}, T)$.
	When $k=2, p=7, q=8$ and $I =(1,2)$, $G_{k,p,q}^I$ is shown in Figure \ref{fig:G25612}.
	Note that $G_{k,p,q}^I$ has a unique longest path of length $k+p+q-1$. It follows easily that $G_{k,p,q}^I$ and $G_{k,p,q}^J$ are not isomorphic for any two different elements $I$ and $J$ of $L_k$. By Proposition \ref{prop:GSAFGSAH}, $G_{k,p,q}^I$ and $G_{k,p,q}^J$ have the same resistance spectrum. Therefore, we can find $2^k$ graphs with the same resistance spectrum.
\end{proof}
\begin{figure}[htbp!]
	\centering
	\tikzstyle{vertex}=[circle, draw, fill=black, inner sep=0, minimum size=4pt]
	\tikzstyle{line}=[line width=1pt,color=black]
	\begin{tikzpicture}[scale=1]
        \draw [line] (6,2)-- (5,2);
        \draw [line] (6,2)-- (7,2);
        \draw [line] (8,1)-- (8,2);
        \draw [line] (8,1)-- (7,2);
        \draw [line] (7,1)-- (6,1);
        \draw [line] (7,1)-- (7,2);
        \draw [line] (5,1)-- (6,1);
        \draw [line] (11,1)-- (12,1);
        \draw [line] (11,1)-- (11,2);
        \draw [line] (10,1)-- (10,2);
        \draw [line] (10,1)-- (9,2);
        \draw [line] (9,1)-- (9,2);
        \draw [line] (12,2)-- (11,2);
        \draw [line] (10,2)-- (11,2);
        \draw [line] (1,3)-- (2,3);
        \draw [line] (2,3)-- (3,3);
        \draw [line] (3,3)-- (4,3);
        \draw [line] (4,3)-- (5,3);
        \draw [line] (5,3)-- (6,3);
        \draw [line] (6,3)-- (7,3);
        \draw [line] (7,3)-- (8,3);
        \draw [line] (8,3)-- (9,3);
        \draw [line] (9,3)-- (10,3);
        \draw [line] (10,3)-- (11,3);
        \draw [line] (11,3)-- (12,3);
        \draw [line] (12,3)-- (13,3);
        \draw [line] (13,3)-- (14,3);
        \draw [line] (14,3)-- (15,3);
        \draw [line] (15,3)-- (16,3);
        \draw [line] (16,3)-- (17,3);
        \draw [line] (8,3)-- (8,2);
        \draw [line] (8,3)-- (7,2);
        \draw [line] (9,2)-- (9,3);
        \draw [line] (10,2)-- (9,3);
        
        \node[vertex] at (6,2) {};
        \node[vertex] at (8,1) {};
        \node[vertex] at (7,1) {};
        \node[vertex] at (5,1) {};
        \node[vertex] at (5,2) {};
        \node[vertex] at (8,2) {};
        \node[vertex] at (6,1) {};
        \node[vertex] at (7,2) {};
        \node[vertex] at (11,1) {};
        \node[vertex] at (10,1) {};
        \node[vertex] at (9,1) {};
        \node[vertex] at (12,2) {};
        \node[vertex] at (12,1) {};
        \node[vertex] at (10,2) {};
        \node[vertex] at (9,2) {};
        \node[vertex] at (11,2) {};
        \node[vertex] at (1,3) {};
        \node[vertex] at (2,3) {};
        \node[vertex] at (3,3) {};
        \node[vertex] at (4,3) {};
        \node[vertex] at (5,3) {};
        \node[vertex] at (6,3) {};
        \node[vertex] at (7,3) {};
        \node[vertex] at (8,3) {};
        \node[vertex] at (9,3) {};
        \node[vertex] at (10,3) {};
        \node[vertex] at (11,3) {};
        \node[vertex] at (12,3) {};
        \node[vertex] at (13,3) {};
        \node[vertex] at (14,3) {};
        \node[vertex] at (15,3) {};
        \node[vertex] at (16,3) {};
        \node[vertex] at (17,3) {};

	\end{tikzpicture}
	\caption{$G_{2,7,8}^{I}$}\label{fig:G25612}
\end{figure}
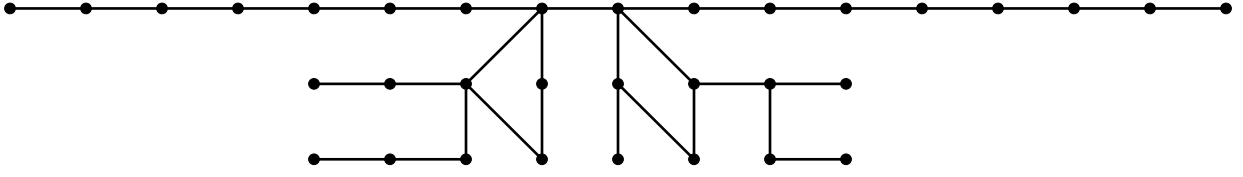

\begin{definition} [$S$-\textit{resistance transitive}]
	Let $G$ be a graph with vertex set $\{g_1, g_2, \ldots, g_n\}$.
	Let $S=\{g_{k_1}, g_{k_2}, \ldots, g_{k_s}\}$ and $T=V (G) \setminus S$, where $3 \leq s \leq n$ and $1 \leq k_1 < \cdots < k_s \leq n$.
	
	Then $G$ is $S$-\textit{resistance transitive} if the following properties are satisfied:\\
	(1) Let $u$ and $v$ be any two vertices in $S$. If $T$ is not an empty set, then for each vertex $x$ in $T$, $R_G (x, u) = R_G (x, v)$.\\
	(2) For each pair of different vertices $u$ and $v$ in $S$, the resistance distance between them equal to the same value.
\end{definition}

\begin{example}
	Let $G$ be a complete graph $K_4$, $G^{\prime}$ an empty graph $\overline {K_4}$, $V(G)=\{g_1, g_2, g_3, g_4\}$ and $V(G^{\prime})=\{g_1^{\prime}, g_2^{\prime}, g_3^{\prime}, g_4^{\prime}\}$.
	Let $S=\{g_1, g_2, g_3\}, T=\{g_4\}, S^{\prime}=\{g_1^{\prime}, g_2^{\prime}, g_3^{\prime}\}, T^{\prime}=\{g_4^{\prime}\}$. Clearly, $R(g_i, g_j) =\frac{1}{2}$ and $R(g_i^{\prime}, g_j^{\prime}) =+\infty$ for $1 \leq i \neq j \leq 4$.
	Thus, $G$ and $G^{\prime}$ are $S$-resistance transitive and $S^{\prime}$-resistance transitive, respectively.
	
\end{example}

\section{Main results}\label{sec:Main}

\begin{lemma}\label{lemma:2}
	Let $G$ be a $S$-resistance transitive graph with $V(G)=\{g_1, g_2, \ldots, g_{n_1}\}$, $S=\{g_1, g_2, \ldots, g_s\}$ and $T=\{g_{s+1}, g_{s+2}, \ldots, g_{n_1}\}$.
	There exists a constant $c$ such that for any two vertices $s_1 \neq s_2 \in S, $
	$R_G (s_1, s_2)=c$.
	Let $H_1, H_2, \ldots, H_t$ be $t$ graphs with $V(H_i) =\{h_{i, 1}, h_{i, 2}, \ldots, h_{i, n_i} \}$ and $H_i \mathcal{U} H_j$ for $i,j \in \{1, 2, \ldots, t\}$. Assume $\RSV(H_i, h_{i,1}) = \RSV(H_j,h_{j,1})$.
	Let $\mathcal{H} = (H_1, H_2, \ldots, H_t)$ and $T_1=\{h_{1,1},h_{2,1},\ldots,h_{t,1}\}$.
	Let $A=\{a_1, a_2, \ldots, a_p\}$ be a partition of a positive integer $t, $ where $p \leq s$.
	
	Then the resistance spectrum of $G (S, A, \mathcal{H}, T_1)$ is
\begin{align*}
	\RS (G)
	\bigcup t\cdot \RS (H_1)
	&\bigcup \{[\RSV (H_1,h_{1,1})+\RSV (H_1,h_{1,1})]^{\sum_{i=1}^{p} \frac{a_i (a_i-1)} {2}} \}\\
	&\bigcup \{ [\RSV (H_1,h_{1,1})+\RSV (H_1,h_{1,1})+c]^{\sum_{1 \leq i < j \leq p} a_i a_j} \}\\
	&\bigcup \{ [\RSV (H_1, h_{1,1})+c]^{ts-t} \} \bigcup \left\{\left[ R_{G}(g_{i}, g_{1})+\RSV (H_1, h_{1,1}) \right]^t \mid i=s+1, \ldots, n \right\}.
\end{align*}
\begin{proof}
	The proof is obvious.
\end{proof}
\end{lemma}

\begin{theorem}\label{theorem:2}
	Let $G$ be a $S$-resistance transitive graph with $V(G)=\{g_1, g_2, \ldots, g_{n_1}\}$, $S=\{g_1, g_2, \ldots, g_s\}$ and $T=\{g_{s+1}, g_{s+2}, \ldots, g_{n_1}\}$.
	There exists a constant $c$ such that for any two vertices $s_1 \neq s_2 \in S, $
	$R_G (s_1, s_2)=c.$ Let $H_1, H_2, \ldots, H_t$ be $t$ graphs with $V(H_i) =\{h_{i, 1}, h_{i, 2}, \ldots, h_{i, n_i} \}$ and $H_i \mathcal{U} H_j$ for $i,j \in \{1, 2, \ldots, t\}$. Assume $\RSV(H_i, h_{i,1}) = \RSV(H_j, h_{j,1})$.
	Let $\mathcal{H} = (H_1, H_2, \ldots, H_t)$  and $T_1=\{h_{1,1},h_{2,1},\ldots$,\\$h_{t,1}\}$.
	Let $A=\{a_1, a_2, \ldots, a_p\}$ and $B=\{b_1, b_2, \ldots, b_q\}$ be two partitions of a positive integer $t$, where $p, q \leq s$. If $A$ and $B$ are of equal sums of squares, then graphs $G (S, A, \mathcal{H}, T_1)$ and $G (S, B, \mathcal{H}, T_1)$ have the same resistance spectrum.
\end{theorem}
\begin{proof}
	By Lemma \ref{lemma:2}, we have
	the resistance spectrum of $G (S, A, \mathcal{H}, T_1)$ is
	\begin{align*}
	\RS (G)
	\bigcup t\cdot \RS (H_1)
	&\bigcup \{[\RSV (H_1,h_{1,1})+\RSV (H_1,h_{1,1})]^{\sum_{i=1}^{p} \frac{a_i (a_i-1)} {2}} \}\\
	&\bigcup \{ [\RSV (H_1,h_{1,1})+\RSV (H_1,h_{1,1})+c]^{\sum_{1 \leq i < j \leq p} a_i a_j} \}\\
	&\bigcup \{ [\RSV (H_1, h_{1,1})+c]^{ts-t} \} \bigcup \left\{\left[ R_{G}(g_{i}, g_{1})+\RSV (H_1, h_{1,1}) \right]^t \mid i=s+1, \ldots, n \right\}.
	\end{align*}
	and
	the resistance spectrum of $G (S, B, \mathcal{H}, T_1)$ is
	\begin{align*}
	\RS (G)
	\bigcup t\cdot \RS (H_1)
	&\bigcup \{[\RSV (H_1,h_{1,1})+\RSV (H_1,h_{1,1})]^{\sum_{i=1}^{q} \frac{b_i (b_i-1)} {2}} \}\\
	&\bigcup \{ [\RSV (H_1,h_{1,1})+\RSV (H_1,h_{1,1})+c]^{\sum_{1 \leq i < j \leq q} b_i b_j} \}\\
	&\bigcup \{ [\RSV (H_1, h_{1,1})+c]^{ts-t} \} \bigcup \left\{\left[ R_{G}(g_{i}, g_{1})+\RSV (H_1, h_{1,1}) \right]^t \mid i=s+1, \ldots, n \right\}.
	\end{align*}	
	Since $A$ and $B$ are of equal sums of squares, we have
	$$\sum_{i=1}^{p} \frac{a_i (a_i-1)} {2}={\sum_{i=1}^{q} \frac{b_i (b_i-1)} {2}}$$
	and
	$$\sum_{1 \leq i < j \leq p} a_i a_j={\sum_{1 \leq i < j \leq q} b_i b_j}.$$
	Thus graphs $G (S, A, \mathcal{H}, T_1)$ and $G (S, B, \mathcal{H}, T_1)$ have the same resistance spectrum.
\end{proof}

\begin{example} \label{exam:n9}
	Let $\mathcal{H} = (H_1, H_2, \ldots, H_6)$, where $H_i$ is a path of length $1$ with $V(H_i)=\{h_{i, 1}, h_{i, 2}\}$ for $i=1, 2, \ldots, 6$.
	Let $T=\{h_{1,1},h_{2,1},\ldots,h_{6,1}\}$, $A_1=\{3, 3\}, A_2=\{4, 1, 1\}$. $A_1$ and $A_2$ are of equal sums of squares.
	Let $S_1 = \overline{K_{3}}$ and $S_2 = K_{3}$. Clearly, $S_1$ and $S_2$ are $V(S_1)$-resistance transitive and $V(S_2)$-resistance transitive, respectively.
	Consequently, the graphs $S_1(V(S_1), A_1, \mathcal{H}, T)$ and $S_1 (V(S_1), A_2, \mathcal{H}, T)$ have the same resistance spectrum, as shown in Pair 1 of Figure \ref{fig:Counterexamples13}; the graphs $S_2(V(S_2), A_1, \mathcal{H}, T)$ and $S_2(V(S_2), A_2, \mathcal{H}, T)$ have the same resistance spectrum, as depicted in Pair 2 of Figure \ref{fig:Counterexamples13}.
	
\end{example}

\begin{proposition}\label{proposition:2}
	Let $G$ be any graph of order $n$, $X$ be a subset of $V (G)$, and $S$ be a complete graph $K_r$ or an empty graph $\overline{K_r}$. $H$ is obtained from $G$ and $S$ by join every vertex of $X$ to every vertex of $S$. Then $H$ is $V(S)$-resistance transitive graph.\\
	\begin{proof}
		The proof is obvious and omitted.
	\end{proof}
\end{proposition}

\begin{theorem}\label{thm:pn9}
	If $n \geq 10$, then there are at least $2(n-9)p(n-9)$ pairs of graphs of order $n$ with the same resistance spectrum, where $p(n-9)$ is the number of partitions of the integer $n-9$.
\end{theorem}
\begin{proof}
	Set $t = n-9$.
	Let $C_1, C_2, \ldots, C_{p(t)}$ be all partitions of $t$ with $C_i=\{c_{i,1}, c_{i,2}, \ldots, c_{i,q_i}\}$, where $c_{i, 1} \geq c_{i, 2} \geq \cdots \geq c_{i, q_i}$, let $G_i$ be a graph with $t$ vertices, where the vertex set is $\{g_{i,1}, g_{i,2}, \ldots, g_{i,t}\}$, $i \in \{1, 2, \ldots, p(t)\}$. The edges of $G_i$ are defined as follows: The first $c_{i,1}$ vertices form a complete subgraph, then the next $c_{i,2}$ vertices form a complete subgraph, and so on, the last $c_{i,q_i}$ vertices form a complete subgraph finally.
	Let $X_{i, j} = \{g_{i, 1}, g_{i, 2}, \ldots, g_{i, j} \}$, $S_1 = K_{3}$ and $S_2 =  \overline{K_{3}}$.	Let $Q_{i, j, k}$ be a graph obtained from $G_{i}$ and $S_k$ by join every vertex of $X_{i, j}$ to every vertex of $S_k$, where $i \in \{1, 2, \ldots, p(t)\}$, $j \in \{1, 2, \ldots, t\}$ and $k \in \{1, 2\}$.
	Thus, by Proposition \ref{proposition:2}, $Q_{i, j, k}$ is $V (S_k)$-resistance transitive graph.

	Let $\mathcal{H} = (H_1, H_2, \ldots, H_6)$, where $H_i$ is a path of length $1$ with $V(H_i)=\{h_{i, 1}, h_{i, 2}\}$ for $i=1, 2, \ldots, 6$.
	Let $T=\{h_{1,1},h_{2,1},\ldots,h_{6,1}\}$.
	$A_1:=\{3, 3\}$ and $A_2:=\{4, 1, 1\}$ are two different partitions of a positive integer $6$, since $A_1$ and $A_2$ are of equal sums of squares,
	by Theorem \ref{theorem:2}, graphs $Q_{i, j, k} (V(S_k), A_1, \mathcal{H}, T)$ and $Q_{i, j, k} (V(S_k), A_2, \mathcal{H}, T)$ have the same resistance spectrum.
	Note that there exists a unique connected component $Q$ in $Q_{i, j, k} (V(S_k), A_l, \mathcal{H}, T)$ satisfying the conditions: $Q$ contains at least $6$ vertices with degree $1$, and there exist three vertices with degree $1$ in $Q$ that share a common neighbor.
	It follows easily that $Q_{i_1, j_1, k_1} (V(S_{k_1}), A_{l_1}, \mathcal{H}, T)$ and $Q_{i_2, j_2, k_2} (V(S_{k_2}), A_{l_2}, \mathcal{H}, T)$ are not isomorphic if $i_1=i_2$, $j_1=j_2$, $k_1=k_2$, and $l_1=l_2$ are not all simultaneously satisfied.
	
	According to the multiplication principle, there are at least $2tp(t)$ pair graphs with the same resistance spectrum.
	Therefore, when $n\ge 10$, there are at least $2 \cdot (n-9) p(n-9)$ pairs of graphs with the same resistance spectrum.	
\end{proof}

\begin{example}
	There exist $8$ pairs of graphs of order $11$ with the same resistance
	spectrum, as shown in Figures \ref{fig:Pair1}--\ref{fig:Pair8}.
	Here $Q_{i, j, k, l} = Q_{i, j, k} (V(S_k), A_l, \mathcal{H}, T)$, where $Q_{i, j, k} (V(S_k), A_l, \mathcal{H}, T)$ is defined as in Theorem \ref{thm:pn9} for $i, j, k, l \in \{1,2\}$.
\end{example}
\tikzstyle{vertex}=[circle, draw, fill=black, inner sep=0, minimum size=4pt]
\tikzstyle{vertex2}=[circle, draw, fill=black, inner sep=0, minimum size=2pt]
\tikzstyle{line}=[line width=1pt,color=black]
\tikzstyle{line2}=[line width=1pt,color=black]
\tikzstyle{line3}=[line width=1pt]
\newcommand{\VERTEX}{
	\node[vertex] (1) at (2.5,3) {};
	\node[vertex] (2) at (3.5,3) {};
	\node[vertex2] (3) at (4.,2.5) {};
	\node[vertex2] (4) at (2.,2.5) {};
	\node[vertex2] (5) at (3.,2) {};
	\node[vertex2] (6) at (1.5,1.5) {};
	\node[vertex2] (7) at (2,1.5) {};
	\node[vertex2] (8) at (2.5,1.5) {};
	\node[vertex2] (9) at (3.5,1.5) {};
	\node[vertex2] (10) at (4,1.5) {};
	\node[vertex2] (11) at (4.5,1.5) {};
}
\newcommand{\VERTEXb}{
	\node[vertex] (1) at (2.5,3) {};
	\node[vertex] (2) at (3.5,3) {};
	\node[vertex2] (3) at (4.5,2.5) {};
	\node[vertex2] (4) at (3,2) {};
	\node[vertex2] (5) at (1.5,2.5) {};
	\node[vertex2] (6) at (2.25,1.5) {};
	\node[vertex2] (7) at (2.75,1.5) {};
	\node[vertex2] (8) at (3.25,1.5) {};
	\node[vertex2] (9) at (3.75,1.5) {};
	\node[vertex2] (10) at (4.5,1.5) {};
	\node[vertex2] (11) at (1.5,1.5) {};
}
\newcommand{\VERTEXc}{
	\node[vertex] (1) at (2.5,3) {};
	\node[vertex] (2) at (3.5,3) {};
	\node[vertex2] (3) at (4.,2.25) {};
	\node[vertex2] (4) at (2.,2.25) {};
	\node[vertex2] (5) at (3.,2.25) {};
	\node[vertex2] (6) at (1.5,1.5) {};
	\node[vertex2] (7) at (2,1.5) {};
	\node[vertex2] (8) at (2.5,1.5) {};
	\node[vertex2] (9) at (3.5,1.5) {};
	\node[vertex2] (10) at (4,1.5) {};
	\node[vertex2] (11) at (4.5,1.5) {};
}
\newcommand{\VERTEXd}{
	\node[vertex] (1) at (2.5,3) {};
	\node[vertex] (2) at (3.5,3) {};
	\node[vertex2] (3) at (4.5,2.25) {};
	\node[vertex2] (4) at (3,2.25) {};
	\node[vertex2] (5) at (1.5,2.25) {};
	\node[vertex2] (6) at (2.25,1.5) {};
	\node[vertex2] (7) at (2.75,1.5) {};
	\node[vertex2] (8) at (3.25,1.5) {};
	\node[vertex2] (9) at (3.75,1.5) {};
	\node[vertex2] (10) at (4.5,1.5) {};
	\node[vertex2] (11) at (1.5,1.5) {};
}
\begin{figure}[htpb!]
	\centering
	\begin{minipage}[b]{0.48\textwidth}
		\begin{subfigure}[b]{0.48\textwidth}
			\centering
			\begin{tikzpicture}
			\VERTEX
			\draw [line3] (3) -- (4);
			\draw [line3] (3) -- (5);
			\draw [line3] (4) -- (5);
			\draw [line2] (4) -- (6);
			\draw [line2] (4) -- (7);
			\draw [line2] (4) -- (8);
			\draw [line2] (3) -- (9);
			\draw [line2] (3) -- (10);
			\draw [line2] (3) -- (11);
			\draw [line] (1) -- (3);
			\draw [line] (1) -- (4);
			\draw [line] (1) -- (5);
			\draw [line3,line width=2pt] (1) -- (2);
			\end{tikzpicture}
			\caption{$Q_{1, 1, 1, 1}$}
		\end{subfigure}
		\begin{subfigure}[b]{0.48\textwidth}
			\centering
			\begin{tikzpicture}
			\VERTEXb
			\draw [line3] (3) -- (4);
			\draw [line3] (3) -- (5);
			\draw [line3] (4) -- (5);
			\draw [line2] (4) -- (6);
			\draw [line2] (4) -- (7);
			\draw [line2] (4) -- (8);
			\draw [line2] (4) -- (9);
			\draw [line2] (5) -- (11);
			\draw [line2] (3) -- (10);
			\draw [line] (1) -- (3);
			\draw [line] (1) -- (4);
			\draw [line] (1) -- (5);
			\draw [line3,line width=2pt] (1) -- (2);
			\end{tikzpicture}
			\caption{$Q_{1, 1, 1, 2}$}
		\end{subfigure}
		\caption{$Q_{1, 1, 1, 1}$ and $Q_{1, 1, 1, 2}$} \label{fig:Pair1}
	\end{minipage}
	\hfill
	\begin{minipage}[b]{0.48\textwidth}
		\begin{subfigure}[b]{0.48\textwidth}
			\centering
			\begin{tikzpicture}
			\VERTEX
			\draw [line3] (3) -- (4);
			\draw [line3] (3) -- (5);
			\draw [line3] (4) -- (5);
			\draw [line2] (4) -- (6);
			\draw [line2] (4) -- (7);
			\draw [line2] (4) -- (8);
			\draw [line2] (3) -- (9);
			\draw [line2] (3) -- (10);
			\draw [line2] (3) -- (11);
			\draw [line] (1) -- (3);
			\draw [line] (1) -- (4);
			\draw [line] (1) -- (5);
			\draw [line] (2) -- (3);
			\draw [line] (2) -- (4);
			\draw [line] (2) -- (5);
			\draw [line3,line width=2pt] (1) -- (2);
			\end{tikzpicture}
			\caption{$Q_{1, 2, 1, 1}$}
		\end{subfigure}
		\begin{subfigure}[b]{0.48\textwidth}
			\centering
			\begin{tikzpicture}
			\VERTEXb
			\draw [line3] (3) -- (4);
			\draw [line3] (3) -- (5);
			\draw [line3] (4) -- (5);
			\draw [line2] (4) -- (6);
			\draw [line2] (4) -- (7);
			\draw [line2] (4) -- (8);
			\draw [line2] (4) -- (9);
			\draw [line2] (5) -- (11);
			\draw [line2] (3) -- (10);
			\draw [line] (1) -- (3);
			\draw [line] (1) -- (4);
			\draw [line] (1) -- (5);
			\draw [line] (2) -- (3);
			\draw [line] (2) -- (4);
			\draw [line] (2) -- (5);
			\draw [line3,line width=2pt] (1) -- (2);
			\end{tikzpicture}
			\caption{$Q_{1, 2, 1, 2}$}
		\end{subfigure}
		\caption{$Q_{1, 2, 1, 1}$ and $Q_{1, 2, 1, 2}$} \label{fig:Pair2}
	\end{minipage}
\end{figure}
\begin{figure}[htpb!]
	\centering
	\begin{minipage}[b]{0.48\textwidth}
		\begin{subfigure}[b]{0.48\textwidth}
			\centering
			\begin{tikzpicture}
			\VERTEXc
			\draw [line2] (4) -- (6);
			\draw [line2] (4) -- (7);
			\draw [line2] (4) -- (8);
			\draw [line2] (3) -- (9);
			\draw [line2] (3) -- (10);
			\draw [line2] (3) -- (11);
			\draw [line] (1) -- (3);
			\draw [line] (1) -- (4);
			\draw [line] (1) -- (5);
			\draw [line3,line width=2pt] (1) -- (2);
			\end{tikzpicture}
			\caption{$Q_{1, 1, 2, 1}$}
		\end{subfigure}
		\begin{subfigure}[b]{0.48\textwidth}
			\centering
			\begin{tikzpicture}
			\VERTEXd
			\draw [line2] (4) -- (6);
			\draw [line2] (4) -- (7);
			\draw [line2] (4) -- (8);
			\draw [line2] (4) -- (9);
			\draw [line2] (5) -- (11);
			\draw [line2] (3) -- (10);
			\draw [line] (1) -- (3);
			\draw [line] (1) -- (4);
			\draw [line] (1) -- (5);
			\draw [line3,line width=2pt] (1) -- (2);
			\end{tikzpicture}
			\caption{$Q_{1, 1, 2, 2}$}
		\end{subfigure}
		\caption{$Q_{1, 1, 2, 1}$ and $Q_{1, 1, 2, 2}$} \label{fig:Pair3}
	\end{minipage}
	\hfill
	\begin{minipage}[b]{0.48\textwidth}
		\begin{subfigure}[b]{0.48\textwidth}
			\centering
			\begin{tikzpicture}
			\VERTEXc
			\draw [line2] (4) -- (6);
			\draw [line2] (4) -- (7);
			\draw [line2] (4) -- (8);
			\draw [line2] (3) -- (9);
			\draw [line2] (3) -- (10);
			\draw [line2] (3) -- (11);
			\draw [line] (1) -- (3);
			\draw [line] (1) -- (4);
			\draw [line] (1) -- (5);
			\draw [line] (2) -- (3);
			\draw [line] (2) -- (4);
			\draw [line] (2) -- (5);
			\draw [line3,line width=2pt] (1) -- (2);
			\end{tikzpicture}
			\caption{$Q_{1, 2, 2, 1}$}
		\end{subfigure}
		\begin{subfigure}[b]{0.48\textwidth}
			\centering
			\begin{tikzpicture}
			\VERTEXd
			\draw [line2] (4) -- (6);
			\draw [line2] (4) -- (7);
			\draw [line2] (4) -- (8);
			\draw [line2] (4) -- (9);
			\draw [line2] (5) -- (11);
			\draw [line2] (3) -- (10);
			\draw [line] (1) -- (3);
			\draw [line] (1) -- (4);
			\draw [line] (1) -- (5);
			\draw [line] (2) -- (3);
			\draw [line] (2) -- (4);
			\draw [line] (2) -- (5);
			\draw [line3,line width=2pt] (1) -- (2);
			\end{tikzpicture}
			\caption{$Q_{1, 2, 2, 2}$}
		\end{subfigure}
		\caption{$Q_{1, 2, 2, 1}$ and $Q_{1, 2, 2, 2}$} \label{fig:Pair4}
	\end{minipage}
\end{figure}
\begin{figure}[htpb!]
	\centering
	\begin{minipage}[b]{0.48\textwidth}
		\begin{subfigure}[b]{0.48\textwidth}
			\centering
			\begin{tikzpicture}
			\VERTEX
			\draw [line3] (3) -- (4);
			\draw [line3] (3) -- (5);
			\draw [line3] (4) -- (5);
			\draw [line2] (4) -- (6);
			\draw [line2] (4) -- (7);
			\draw [line2] (4) -- (8);
			\draw [line2] (3) -- (9);
			\draw [line2] (3) -- (10);
			\draw [line2] (3) -- (11);
			\draw [line] (1) -- (3);
			\draw [line] (1) -- (4);
			\draw [line] (1) -- (5);
			\end{tikzpicture}
			\caption{$Q_{2, 1, 1, 1}$}
		\end{subfigure}
		\begin{subfigure}[b]{0.48\textwidth}
			\centering
			\begin{tikzpicture}
			\VERTEXb
			\draw [line3] (3) -- (4);
			\draw [line3] (3) -- (5);
			\draw [line3] (4) -- (5);
			\draw [line2] (4) -- (6);
			\draw [line2] (4) -- (7);
			\draw [line2] (4) -- (8);
			\draw [line2] (4) -- (9);
			\draw [line2] (5) -- (11);
			\draw [line2] (3) -- (10);
			\draw [line] (1) -- (3);
			\draw [line] (1) -- (4);
			\draw [line] (1) -- (5);
			\end{tikzpicture}
			\caption{$Q_{2, 1, 1, 2}$}
		\end{subfigure}
		\caption{$Q_{2, 1, 1, 1}$ and $Q_{2, 1, 1, 2}$} \label{fig:Pair5}
	\end{minipage}
	\hfill
	\begin{minipage}[b]{0.48\textwidth}
		\begin{subfigure}[b]{0.48\textwidth}
			\centering
			\begin{tikzpicture}
			\VERTEX
			\draw [line3] (3) -- (4);
			\draw [line3] (3) -- (5);
			\draw [line3] (4) -- (5);
			\draw [line2] (4) -- (6);
			\draw [line2] (4) -- (7);
			\draw [line2] (4) -- (8);
			\draw [line2] (3) -- (9);
			\draw [line2] (3) -- (10);
			\draw [line2] (3) -- (11);
			\draw [line] (1) -- (3);
			\draw [line] (1) -- (4);
			\draw [line] (1) -- (5);
			\draw [line] (2) -- (3);
			\draw [line] (2) -- (4);
			\draw [line] (2) -- (5);
			\end{tikzpicture}
			\caption{$Q_{2, 2, 1, 1}$}
		\end{subfigure}
		\begin{subfigure}[b]{0.48\textwidth}
			\centering
			\begin{tikzpicture}
			\VERTEXb
			\draw [line3] (3) -- (4);
			\draw [line3] (3) -- (5);
			\draw [line3] (4) -- (5);
			\draw [line2] (4) -- (6);
			\draw [line2] (4) -- (7);
			\draw [line2] (4) -- (8);
			\draw [line2] (4) -- (9);
			\draw [line2] (5) -- (11);
			\draw [line2] (3) -- (10);
			\draw [line] (1) -- (3);
			\draw [line] (1) -- (4);
			\draw [line] (1) -- (5);
			\draw [line] (2) -- (3);
			\draw [line] (2) -- (4);
			\draw [line] (2) -- (5);
			\end{tikzpicture}
			\caption{$Q_{2, 2, 1, 2}$}
		\end{subfigure}
		\caption{$Q_{2, 2, 1, 1}$ and $Q_{2, 2, 1, 2}$} \label{fig:Pair6}
	\end{minipage}
\end{figure}
\begin{figure}[htpb!]
	\centering
	\begin{minipage}[b]{0.48\textwidth}
		\begin{subfigure}[b]{0.48\textwidth}
			\centering
			\begin{tikzpicture}
			\VERTEXc
			\draw [line2] (4) -- (6);
			\draw [line2] (4) -- (7);
			\draw [line2] (4) -- (8);
			\draw [line2] (3) -- (9);
			\draw [line2] (3) -- (10);
			\draw [line2] (3) -- (11);
			\draw [line] (1) -- (3);
			\draw [line] (1) -- (4);
			\draw [line] (1) -- (5);
			\end{tikzpicture}
			\caption{$Q_{2, 1, 2, 1}$}
		\end{subfigure}
		\begin{subfigure}[b]{0.48\textwidth}
			\centering
			\begin{tikzpicture}
			\VERTEXd
			\draw [line2] (4) -- (6);
			\draw [line2] (4) -- (7);
			\draw [line2] (4) -- (8);
			\draw [line2] (4) -- (9);
			\draw [line2] (5) -- (11);
			\draw [line2] (3) -- (10);
			\draw [line] (1) -- (3);
			\draw [line] (1) -- (4);
			\draw [line] (1) -- (5);
			\end{tikzpicture}
			\caption{$Q_{2, 1, 2, 2}$}
		\end{subfigure}
		\caption{$Q_{2, 1, 2, 1}$ and $Q_{2, 1, 2, 2}$} \label{fig:Pair7}
	\end{minipage}
	\hfill
	\begin{minipage}[b]{0.48\textwidth}
		\begin{subfigure}[b]{0.48\textwidth}
			\centering
			\begin{tikzpicture}
			\VERTEXc
			\draw [line2] (4) -- (6);
			\draw [line2] (4) -- (7);
			\draw [line2] (4) -- (8);
			\draw [line2] (3) -- (9);
			\draw [line2] (3) -- (10);
			\draw [line2] (3) -- (11);
			\draw [line] (1) -- (3);
			\draw [line] (1) -- (4);
			\draw [line] (1) -- (5);
			\draw [line] (2) -- (3);
			\draw [line] (2) -- (4);
			\draw [line] (2) -- (5);
			\end{tikzpicture}
			\caption{$Q_{2, 2, 2, 1}$}
		\end{subfigure}
		\begin{subfigure}[b]{0.48\textwidth}
			\centering
			\begin{tikzpicture}
			\VERTEXd
			\draw [line2] (4) -- (6);
			\draw [line2] (4) -- (7);
			\draw [line2] (4) -- (8);
			\draw [line2] (4) -- (9);
			\draw [line2] (5) -- (11);
			\draw [line2] (3) -- (10);
			\draw [line] (1) -- (3);
			\draw [line] (1) -- (4);
			\draw [line] (1) -- (5);
			\draw [line] (2) -- (3);
			\draw [line] (2) -- (4);
			\draw [line] (2) -- (5);
			\end{tikzpicture}
			\caption{$Q_{2, 2, 2, 2}$}
		\end{subfigure}
		\caption{$Q_{2, 2, 2, 1}$ and $Q_{2, 2, 2, 2}$} \label{fig:Pair8}
	\end{minipage}
\end{figure}

\section{Conclusion}\label{sec:Conclusion}
 In this paper, we propose a method for constructing graphs with the same resistance spectrum. We also present a lower bound for the number of pairs of graphs which have the same resistance spectrum among graphs with $n(\ge 10)$ vertices.  Our method is derived from observing pairs 1 and 2 in Figure 1. By carefully examining the other pairs of graphs in Figure 1, it is possible to find more methods for construting non-isomorphic graphs with the same resistance spectrum.



\bibliography{xfpan-2024}

\end{document}